\documentclass[11pt,letterpaper]{amsart}
\usepackage{amsmath}
\usepackage{amstext}
\usepackage{amssymb}
\usepackage{amsfonts}
\usepackage{enumerate}
\usepackage{amsthm}
\usepackage{mathrsfs}
\usepackage[all]{xy}

\newtheorem{theorem}{Theorem}[section]

\newtheorem{proposition}[theorem]{Proposition}
\newtheorem{corollary}[theorem]{Corollary}
\theoremstyle{definition}
\newtheorem{definition}[theorem]{Definition}
\newtheorem{remark}[theorem]{Remark}
\newtheorem{question}[theorem]{Question}
\newtheorem*{theorem*}{Theorem}

\newcommand{\n}[1]{ \left\|#1\right\| }
\newcommand{\bign}[1]{ \big\|#1\big\| }

\newcommand{\N}{{\mathbb{N}}}
\newcommand{\R}{{\mathbb{R}}}

\newcommand{\pair}[2]{{\langle #1, #2 \rangle}}

\newcommand{\fbs}{{\mathscr{F}}}

\DeclareMathOperator{\Lip}{Lip}

\DeclareMathOperator{\conv}{conv}

\numberwithin{equation}{section}

\title{Lipschitz $p$-convex and $q$-concave maps}
\author{Javier Alejandro Ch\'avez-Dom\'inguez}
\address{The University of Texas at Austin,
Mathematics Department,
2515 Speedway Stop C1200,
Austin, Texas 78712.}
\email{jachavezd@math.utexas.edu}
\thanks{Portions of this paper are part of the author's Ph.D. dissertation, prepared at Texas A\&{}M University under the supervision of William B. Johnson.
Partially supported by NSF grants DMS-1001321 and DMS-1400588.}

\begin{document}

\maketitle

\begin{abstract}
The notions of $p$-convexity and $q$-concavity are mostly known because of their importance as a tool in the study of isomorphic properties of Banach lattices, but they also play a role in several results involving linear maps between Banach spaces and Banach lattices.
In this paper we introduce Lipschitz versions of these concepts, dealing with maps between metric spaces and Banach lattices, and start by proving nonlinear versions of two well-known factorization theorems through $L_p$ spaces due to Maurey/Nikishin and Krivine.
We also show that a Lipschitz map from a metric space into a Banach lattice is Lipschitz $p$-convex if and only if its linearization is $p$-convex.
Furthermore, we elucidate why there is such a close relationship between the linear and nonlinear concepts by proving characterizations of Lipschitz $p$-convex and Lipschitz $q$-concave maps in terms of factorizations through $p$-convex and $q$-concave Banach lattices, respectively, in the spirit of the work of Raynaud and Tradacete.
\end{abstract}

\section{Introduction}

The classical examples of Banach spaces of functions or sequences (say, $L_p[0,1]$ or $c_0$) come naturally endowed with an order structure that is compatible with the norm, and this is often a useful tool.
Banach spaces with such ``extra'' order structure are called Banach lattices, and, paraphrasing Lindenstrauss and Tzafriri \cite{Lindenstrauss-Tzafriri-II}, this additional ingredient makes
the theory of Banach lattices in some regards simpler, cleaner and more complete than the theory for general Banach spaces.
For a Banach-space flavored study of the fundamentals of the theory of Banach lattices, the reader is referred to \cite{Lindenstrauss-Tzafriri-II}.
Because of the extra structure present in Banach lattices, the morphisms of interest are not just bounded linear maps but in fact maps that take advantage of the order.
Two of the most important classes of such maps are the $p$-convex and $q$-concave ones.
These two notions play an important role in the study of isomorphic properties of lattices, for example uniform convexity in Banach lattices \cite[Sec. 1.f]{Lindenstrauss-Tzafriri-II} and the study of rearrangement invariant function spaces \cite[Sec. 2.e]{Lindenstrauss-Tzafriri-II}.
Let us recall their definitions. Consider $1 \le p \le \infty$.
A linear map $T : X \to E$ from a Banach space $X$ to a Banach lattice $E$ is called \emph{$p$-convex} if there exists a constant $M < \infty$ such that for all $v_1, \dotsc, v_n \in X$
$$
\n{ \Big( \sum_{j=1}^n |Tv_j|^p \Big)^{1/p} }_E \le M \Big( \sum_{j=1}^n \n{v_j}_X^p \Big)^{1/p}, \quad\text{ if $1\le p < \infty$ }
$$
or
$$
\n{ \bigvee_{j=1}^n |Tv_j| }_E \le M \max_{1 \le j \le n} \n{v_j}_X, \quad\text{ if $p =\infty$. }
$$
The smallest such constant $M$ is denoted $M^{(p)}(T)$.
On the other hand, a linear map $S : E \to Y$ from a Banach lattice $E$ to a Banach space $Y$ is called $q$-concave if there exists a constant $M < \infty$ such that for all $v_1, \dotsc, v_n \in E$,
$$
\Big( \sum_{j=1}^n \n{Sv_j}_Y^q \Big)^{1/q} \le M \n{ \Big( \sum_{j=1}^n |v_j|^q \Big)^{1/q} }_E , \quad\text{ if $1\le q < \infty$ }
$$
or
$$
\max_{1 \le j \le n} \n{Sv_j}_Y \le \n{ \bigvee_{j=1}^n |v_j| }_E, \quad\text{ if $q =\infty$. }
$$
The smallest such constant $M$ is denoted $M_{(q)}(S)$.
The constants $M^{(p)}(T)$ and $M_{(q)}(S)$ are called the \emph{$p$-convexity} constant of $T$, respectively \emph{$q$-concavity} constant of $S$.
When we say that a Banach lattice $E$ is $p$-convex (resp. $q$-concave) we mean that the identity map $I_E : E \to E$ is $p$-convex (resp. $q$-concave).

In this paper we develop nonlinear counterparts of these two concepts,
considering Lipschitz maps between a metric space and a Banach lattice, 
and show how some of
the elementary results from the theory of $p$-convex and $q$-concave maps admit generalizations to the Lipschitz setting.
The rest of the paper is organized as follows.
In Section \ref{sec-preliminaries} we recall some preliminaries that will be needed in the sequence.
In Section \ref{sec-lip-p-convex} we prove a nonlinear version of the Maurey/Nikishin factorization theorem, which suggests our definition of Lipschitz $p$-convex maps.
We then prove that a map is Lipschitz $p$-convex if and only if its canonical linearization is $p$-convex.
In Section \ref{sec-lip-q-concave} we introduce the concept of Lipschitz $q$-concave maps, and prove that the composition of a Lipschitz $p$-convex map followed by a Lipschitz $p$-concave one factors through an $L_p$ space, a nonlinear version of a theorem due to Krivine.
Finally, Section \ref{sec-factorization-theorems} makes clear why there is a close parallel between the linear and nonlinear situations. This is done by proving characterizations of Lipschitz $p$-convex and Lipschitz $q$-concave maps in terms of factorizations through $p$-convex and $q$-concave Banach lattices, respectively.
This is a nonlinear generalization of the work of Raynaud and Tradacete \cite{Raynaud-Tradacete}.

\section{Notation and preliminaries}\label{sec-preliminaries}

For the basic background and notation on Banach lattices, the reader is directed to \cite{Lindenstrauss-Tzafriri-II}.
We use the convention of having \emph{pointed} metric spaces, i.e. with a designated special point always denoted by $0$.
As customary, if $E$ is a Banach space then $B_E$ denotes the closed unit ball of $E$ and $E^*$ its linear dual. If $X$ is a metric space,
$\Lip_0(X,E)$ denotes the Banach space of Lipschitz functions $T : X \to E$ such that $T(0)=0$, with addition defined pointwise and the Lipschitz constant $\Lip(T)$ as the norm of $T$.
We use the shorthand $X^\# := \Lip_0(X,\R)$.

Let us recall the definition and basic properties of the space of Arens and Eells \cite{Arens-Eells-56}.
We follow the presentation in \cite{Weaver}.
A \emph{molecule} on a metric space $X$ is a finitely supported function $m:X\rightarrow \R$ such that $\sum_{x\in X} m(x) =0$.
For $x, x' \in X$ we denote by $m_{xx'}$ the molecule $\chi_{\{x\}}-\chi_{\{x'\}}$.
The simplest molecules, i.e. those of the form $a m_{xx'}$ with $x,x'\in X$ and $a$ a real number are called 
\emph{atoms}.
It is easy to show that every molecule can be expressed as a sum of atoms (for instance, by induction on the cardinality of the support of the molecule).
The Arens-Eells space of $X$, denoted $\fbs(X)$, is the completion of the space of molecules with the norm
\begin{equation}\label{eqn-def-arens-eells-norm}
\n{m}_{\fbs} := \inf \bigg\{ \sum_{j=1}^n |a_j| d(x_j,x'_j) \; : \; m = \sum_{j=1}^n a_j m_{x_jx_j'} \bigg\}.
\end{equation}
The fundamental properties of the Arens-Eells space are summarized in the following theorem \cite{Arens-Eells-56}, \cite[pp. 39-41]{Weaver}.

\begin{theorem}\label{thm-Arens-Eells}
\begin{enumerate}[(i)]
\item
$\n{\cdot}_{\fbs}$ is a norm on the vector space of molecules on $X$.
\item 
The dual of $\fbs(X)$ is (canonically) isometrically isomorphic to $X^\#$. Moreover, on bounded subsets of $X^\#$ the weak$^*$ topology coincides with the topology of pointwise convergence.
\item\label{universal-property}
The map $\iota : x \mapsto m_{x0}$ is an isometric embedding of $X$ into $\fbs(X)$. Moreover,
for any Banach space $E$ and any Lipschitz map $T : X \to E$ with $T(0)=0$ there is a unique linear map $\hat{T} : \fbs(X) \to E$ such that $\hat{T} \circ \iota = T$. Furthermore, $\bign{\hat{T}} = \Lip(T)$.
\end{enumerate}
\end{theorem}

Because of the universal property \eqref{universal-property}, the space $\fbs(X)$ is also called the free Lipschitz space of $X$, or simply the free space of $X$.
These spaces have been recently used as tools in nonlinear Banach space theory, see \cite{Godefroy-Kalton-03}, \cite{Kalton-04} and the survey \cite{Godefroy-Lancien-Zizler}.

\section{Lipschitz $p$-convex maps}\label{sec-lip-p-convex}

The concept of Lipschitz $p$-convex map was inspired by our discovery of the following non-linear version of the Maurey/Nikishin factorization theorem.
The proof presented here follows very closely the presentation of Albiac and Kalton \cite[Thm. 7.1.2]{Albiac-Kalton}.

\begin{theorem}\label{thm-nonlinear-nikishin}
Let $X$ be a metric space and $\mu$ be a $\sigma$-finite measure on some measurable space $(\Omega, \Sigma)$ and $1 \le q < p <\infty$. Suppose that $T : X \to L_q(\mu)$ is a Lipschitz map and $C > 0$.
The following are equivalent:
\begin{enumerate}[(a)]
\item
There exists a density function $h$ on $\Omega$ such that
\begin{equation}\label{eqn-7-1}
\left[ \int_{\{h>0\}} \left|\frac{Tx-Tx'}{h^{1/q}}\right|^ph\,d\mu \right]^{1/p} \le C d(x,x'), \qquad x,x' \in X
\end{equation}
and
\begin{equation}\label{eqn-7-2}
\mu\{ |Tx-Tx'| >0, h=0 \} = 0 \qquad x, x' \in X.
\end{equation}
\item For every $x_1,\dots,x_n,x_1',\dots,x_n' \in X$ and $\lambda_1,\dots,\lambda_n \ge 0$,
\begin{equation}\label{eqn-7-3}
\n{ \bigg(\sum_{j=1}^n \lambda_j|Tx_j-Tx_j'|^p\bigg)^{1/p} }_{L_q(\mu)} \le C \bigg(\sum_{j=1}^n \lambda_j d(x_j,x_j')^p\bigg)^{1/p}
\end{equation}
\end{enumerate}
\end{theorem}

As in the linear case, condition $(a)$ is equivalent to the existence of a factorization diagram
\begin{equation}\label{eqn-factorization-maurey-nikishin-nonlinear}
	\xymatrix{
	X \ar[r]^{T}\ar[d]^S &L_{q}(\mu) \\
	L_p(hd\mu) \ar[r]^{i_{p,q}} & L_q(hd\mu) \ar[u]^{j}
	}	
\end{equation}
where $S$ is a Lipschitz function with $\Lip(S)\le C$ and the isometry $j$ has, in fact, range $L_q(A,\mu)$ where $A = \{ h > 0\}$.
Also, if we consider $X$ as a pointed metric space with a designated point $0\in X$ and impose the condition $T(0)=0$, condition \eqref{eqn-7-2} can be replaced by the somewhat simpler one
$$
\mu\{ |Tx| > 0, h=0 \} = 0 \qquad x \in X\setminus\{0\}.
$$

\begin{proof}[Proof of Theorem \ref{thm-nonlinear-nikishin}]
Without loss of generality, via a first change of density, we may assume that $\mu$ is in fact a probability measure.

$(a) \Rightarrow (b)$
Let $x_1,\dots,x_n,x_1',\dots,x_n' \in X$ and $\lambda_1,\dots,\lambda_n \ge 0$. Since $\mu$ is a probability space and $q<p$, the $L_q(hd\mu)$ norm is smaller than the $L_p(hd\mu)$ norm and thus
\begin{align*}
\n{ \bigg(\sum_{j=1}^n \lambda_j|Tx_j-Tx_j'|^p\bigg)^{1/p} }_{L_q(\mu)} &= \n{ \bigg(\sum_{j=1}^n \lambda_j|Sx_j-Sx_j'|^p\bigg)^{1/p} }_{L_q(hd\mu)} \\
&\le \n{ \bigg(\sum_{j=1}^n \lambda_j|Sx_j-Sx_j'|^p\bigg)^{1/p} }_{L_p(hd\mu)} \\
&= \bigg( \sum_{j=1}^n \lambda_j\n{Sx_j-Sx_j'}_{L_p(hd\mu)}^p\bigg)^{1/p}\\
&\le C \bigg( \sum_{j=1}^n \lambda_jd(x_j,x_j')^p\bigg)^{1/p}.
\end{align*}

$(b) \Rightarrow (a)$
Assume $C$ is the best constant in \eqref{eqn-7-3}. Without loss of generality, we can assume $C=1$ (by considering $T/C$ instead of $T$).

Let
$$
W_0 = \left\{ f : \Omega \to \R \,:\, 0 \le f \le  \left(\sum_{j=1}^n \lambda_j|Tx_j-Tx_j'|^p\right)^{q/p}, \;\; \sum_{j=1}^n \lambda_j d(x_j,x_j')^p \le 1 \right\},
$$
and let $W$ be the closure of $W_0$ in $L_1(\mu)$. 
Since $1$ is the best constant in \eqref{eqn-7-3},
\begin{equation}\label{eqn-sup}
\sup \left\{ \int_\Omega f \,d\mu : f\in W_0 \right\} = \sup \left\{ \int_\Omega f \,d\mu : f\in W \right\}  =1.
\end{equation}

CLAIM 1: $W^{p/q}$ is a convex set. 

It suffices to show that $W_0^{p/q}$ is a convex set. Let $f,g \in W_0$ and $a,b \ge 0$ with $a+b=1$.
From the definition of $W_0$, there exist $x_1,\dots,x_n,x_1',\dots,x_n' \in X$ and $\lambda_1,\dots,\lambda_n \ge 0$ such that
$$
0 \le f \le \left(\sum_{j=1}^n \lambda_j|Tx_j-Tx_j'|^p\right)^{q/p} \quad \text{and} \quad \sum_{j=1}^n \lambda_j d(x_j,x_j')^p \le 1,
$$
and there also exist 
$y_1,\dots,y_m,y_1',\dots,y_m' \in X$ and $\sigma_1,\dots,\sigma_m \ge 0$ such that
$$
0 \le g \le \left(\sum_{k=1}^m \sigma_k|Ty_k-Ty_k'|^p\right)^{q/p} \quad \text{and} \quad \sum_{k=1}^m \sigma_k d(y_k,y_k')^p \le 1.
$$
Now
\begin{align*}
0 \le \big( af^{p/q}+b^{p/q} \big)^{q/p} &\le \left( a\sum_{j=1}^n \lambda_j|Tx_j-Tx_j'|^p + b\sum_{k=1}^m \sigma_k|Ty_k-Ty_k'|^p \right)^{q/p} \\
&\le \left( \sum_{j=1}^n a\lambda_j|Tx_j-Tx_j'|^p + \sum_{k=1}^m b\sigma_k|Ty_k-Ty_k'|^p \right)^{q/p},
\end{align*}
and since
$$
\sum_{j=1}^n a\lambda_jd(x_j,x_j')^p + \sum_{k=1}^m b\sigma_kd(y_k,y_k')^p \le a + b =1,
$$
we conclude that $\big( af^{p/q}+b^{p/q} \big)^{q/p} \in W_0$ and therefore $W_0^{p/q}$ is a convex set.

CLAIM 2: There exists $h \in W$ such that $\int hd\mu = 1$.

Since $\mu$ is a probability measure, the map $f \mapsto \int fd\mu$ is a continuous linear functional and therefore it will suffice to show that $W$ is a weakly compact set in $L_1(\mu)$.
By definition, $W$ is norm closed. Moreover, it is convex
so $W$ is weakly closed.

In order to show that $W$ is weakly compact, all that is left to check is equi-integrability. Suppose that $W$ is not equi-integrable. Then there exist $\delta>0$, a sequence $(E_n)_{n=1}^\infty$ of disjoint subsets of $\Omega $ and a sequence $(f_n)_{n=1}^\infty$ in $W$ such that for all $n\in\N$, 
$$
\int_{E_n} f_n\,d\mu >\delta.
$$
Thus given any $N\in \N$, since the sets $(E_n)$ are disjoint,
\begin{align*}
N\delta &\le \sum_{n=1}^N \int_{E_n} f_n\,d\mu \le \int \max\{f_1,\dots,f_n\}\,d\mu \\
&\le N^{q/p} \int \left( \sum_{n=1}^N \frac{1}{N} f_n^{p/q} \right)^{q/p}\,d\mu.
\end{align*}
By Claim 1, this last integral is at most $1$, so $\delta \le N^{q/p-1}$. Since $q/p<1$, this is a contradiction for large enough $N$. 

Now, let $f \in W$ and $\tau>0$. By Claim 1,
$$
\frac{1}{1+\tau}\big( h^{p/q} + \tau f^{p/q} \big) \in W^{p/q},
$$
so from \eqref{eqn-sup}
\begin{equation}\label{eqn-variation}
(1+\tau)^{q/p} \ge \int \big( h^{p/q} + \tau f^{p/q} \big) ^{q/p} \,d\mu.
\end{equation}
But
\begin{multline*}
\int \big( h^{p/q} + \tau f^{p/q} \big) ^{q/p} \,d\mu. \ge \int_{\{h>0\}} h\,d\mu + \tau^{q/p}\int_{\{h=0\}} f\,d\mu
= 1 + \tau^{q/p} \int_{\{h=0\}} f\,d\mu.
\end{multline*}
so, since $0<q/p<1$,
$$
0 \le \int_{\{h=0\}} f\,d\mu \le \frac{(1+\tau)^{q/p}-1}{\tau^{q/p}} \underset{\tau\rightarrow 0^+}{\longrightarrow} 0,
$$
from where we get \eqref{eqn-7-2}.
by considering $f$ of the form $|Tx-Tx'|/d(x,x')$.

From \eqref{eqn-variation},
\begin{equation}\label{eqn-fatou}
\frac{(1+\tau)^{q/p}-1}{\tau} \ge \int_{\{h>0\}} \left[ \frac{ \big( 1+\tau(f/h)^{p/q} \big)^{q/p} -1 }{\tau} \right] hd\mu.
\end{equation}
Letting $\tau\rightarrow 0^+$, the left-hand side of \eqref{eqn-fatou} converges to $q/p$.
By Fatou's lemma, the right-hand side is at least
$$
\frac{q}{p} \int_{\{h>0\}} \left( \frac{f}{h} \right)^{p/q} hd\mu.
$$
By considering once more $f$ of the form $|Tx-Tx'|/d(x,x')$ we get \eqref{eqn-7-1}.
\end{proof}

Note that the equivalent conditions in Theorem \ref{thm-nonlinear-nikishin} are in fact equivalent to the $p$-convexity of the linear extension $\hat{T} : \fbs(X) \to L_q(\mu)$ of $T : X \to L_q(\mu)$. If we have a Lipschitz factorization as in \eqref{eqn-factorization-maurey-nikishin-nonlinear}, that immediately gives a linear factorization
$$
\xymatrix{
\fbs(X) \ar[r]^{\hat{T}}\ar[d]^{\hat{S}} &L_{q}(\mu) \\
L_p(hd\mu) \ar[r]^{i_{p,q}} & L_q(hd\mu) \ar[u]^{j}
}
$$
and vice versa; such a linear factorization of $\hat{T}$ is equivalent to the $p$-convexity of $\hat{T}$ by the classical Maurey/Nikishin theorem. 
Hence, the following definition is a natural one:

\begin{definition}
Let $1 \le p \le \infty$.
Let $X$ be a metric space and $E$ a Banach lattice. A Lipschitz map $T : X \to E$ is called \emph{Lipschitz $p$-convex} if there exists a constant $C\ge 0$ 
for any $x_j,x_j' \in X$ and $\lambda_j \ge 0$,
$$
\n{ \bigg(\sum_{j=1}^n \lambda_j|Tx_j-Tx_j'|^p\bigg)^{1/p} }_E \le C \bigg(\sum_{j=1}^n  \lambda_jd(x_j,x_j')^p\bigg)^{1/p},
$$
(with the obvious adjustment if $p=\infty$).
The smallest such constant $C$ is called the \emph{Lipschitz $p$-convexity constant} of $T$ and is denoted by $M^{(p)}_{\Lip}(T)$.
\end{definition}

Note that this is a generalization of the linear concept: a linear map $T : X \to E$ from a Banach space $X$ to a Banach lattice $E$ is $p$-convex if and only if it is Lipschitz $p$-convex, and with the same constant. 

One could be tempted to follow the footsteps of \cite{Farmer-Johnson-09} and ``get rid of the constants'' in the Lipschitz $p$-convexity condition; that is, checking that it suffices to have the condition with all $\lambda_j$ being equal to 1.
For that to be true we need extra continuity conditions on the lattice, so we do not take such simplification as the definition of Lipschitz $p$-convexity.

The situation of Theorem \ref{thm-nonlinear-nikishin}, where a Lipschitz map turned out to be Lipschitz $p$-convex if and only if its linearization is $p$-convex, is in fact the general case as demonstrated below.

\begin{theorem}\label{theorem-lipschitz-convexity}
Let $X$ be a metric space and $E$ a Banach lattice. A Lipschitz map $T : X \to E$ is Lipschitz $p$-convex if and only if $\hat{T} : \fbs(X) \to E$ is $p$-convex. Moreover, in this case the $p$-convexity constants are the same.
\end{theorem}

\begin{proof}
The ``if'' part is trivial: $p$-convexity of $\hat{T}$ clearly implies Lipschitz $p$-convexity of $T$ with no increment in the constant, since $\n{m_{xx'}}_{\fbs(X)} = d(x,x')$ and $\hat{T}m_{xx'} = Tx-Tx'$.

Now suppose that $T$ is Lipschitz $p$-convex.
Let $\varphi_j^* \in E^*$ be arbitrary.
For any $x_j,x_j' \in X$ with $x_j \not=x'_j$ we obviously have
$$
\bigg( \sum_j \Big| \frac{\pair{\varphi_j^*}{Tx_j-Tx'_j}}{d(x_j,x'_j)} \Big|^{p'} \bigg)^{1/p'} 
= \sup_{ \sum_j |\alpha_j|^p \le 1} \sum_j \alpha_j \frac{\pair{\varphi_j^*}{Tx_j-Tx'_j}}{d(x_j,x'_j)}.
$$
Using \cite[Prop. 1.d.2.(iii)]{Lindenstrauss-Tzafriri-II}, the latter is bounded by
\begin{multline*}
\sup_{ \sum_j |\alpha_j|^p \le 1} \bigg( \bigg( \sum_j |\varphi_j^*|^{p'} \bigg)^{1/p'} \bigg) \bigg( \bigg( \sum_j |\alpha_j|^p \frac{\big|Tx_j-Tx'_j\big|^p}{d(x_j,x'_j)^p} \bigg)^{1/p} \bigg) \\
\le \n{ \bigg( \sum_j |\varphi_j^*|^{p'} \bigg)^{1/p'} }_{L^*} \sup_{ \sum_j |\alpha_j|^p \le 1} \n{ \bigg( \sum_j |\alpha_j|^p \frac{\big|Tx_j-Tx'_j\big|^p}{d(x_j,x'_j)^p} \bigg)^{1/p} }_E 
\end{multline*}
The Lipschitz $p$-convexity of $T$ allows us to bound this by
\begin{multline*}
\n{ \bigg( \sum_j |\varphi_j^*|^{p'} \bigg)^{1/p'} }_{E^*} M^{(p)}_{\Lip}(T)\sup_{ \sum_j |\alpha_j|^p \le 1} \bigg( \sum_j |\alpha_j|^p \frac{d(x_j,x'_j)^p}{d(x_j,x'_j)^p} \bigg)^{1/p} \\
= M^{(p)}_{\Lip}(T) \n{ \bigg( \sum_j |\varphi_j^*|^{p'} \bigg)^{1/p'} }_{E^*}.
\end{multline*}
Therefore,
$$
\bigg( \sum_j \Big| \frac{(\hat{T}^*\varphi_j^*)(x_j)-(\hat{T}^*\varphi_j^*)(x'_j)}{d(x_j,x'_j)} \Big|^{p'} \bigg)^{1/p'}  \le  M^{(p)}_{\Lip}(T) \n{ \bigg( \sum_j |\varphi_j^*|^{p'} \bigg)^{1/p'} }_{E^*},
$$
so taking the supremum over all pairs $x_j,x'_j \in X$ with $x_j \not=x'_j$ we conclude
$$
\bigg( \sum_j \n{ \hat{T}^*\varphi_j^* }_{X^\#}^{p'} \bigg)^{1/p'}  \le  M^{(p)}_{\Lip}(T) \n{ \bigg( \sum_j |\varphi_j^*|^{p'} \bigg)^{1/p'} }_{E^*}.
$$
Since the $\varphi_j^* \in L^*$ were arbitrary, this means that $\hat{T}^* : L^* \to X^\#$ is $p'$-concave with $M_{(p')}(\hat{T}^*) \le M^{(p)}_{\Lip}(T)$, and by duality \cite[Prop. 1.d.4]{Lindenstrauss-Tzafriri-II}
 $\hat{T} : \fbs(X) \to L$ is $p$-convex with $M^{(p)}(\hat{T}) \le M^{(p)}_{\Lip}(T)$.
\end{proof}

Let us note that the argument in the previous result is based on the duality between $p$-convexity and $p'$-concavity, so it seems unlikely that it could be used to prove a similar result for other classes of maps obtained by replacing the expression $ \Big(\sum_j |x_j|^p\Big)^{1/p}$ by other homogeneous functions given by the Krivine functional calculus for Banach lattices.

\begin{remark}
	It is worth pointing out that Theorem \ref{thm-nonlinear-nikishin} follows immediately from Theorem \ref{theorem-lipschitz-convexity} and the classical Maurey/Nikishin factorization theorem \cite[Thm. 7.1.2]{Albiac-Kalton}.
\end{remark}

\section{Lipschitz $q$-concave maps}\label{sec-lip-q-concave}

Following up on the previous work on $p$-convexity we now point our attention to the natural companion concept, that of Lipschitz $q$-concavity.

\begin{definition}
Let $X$ be a metric space and $E$ a Banach lattice. A map $T : E \to X$ is called \emph{Lipschitz $q$-concave} if there exists a constant $C\ge 0$ such that for any $v_j,v'_j \in E$ and any $\lambda_j \ge 0$,
$$
\bigg(\sum_{j=1}^n  \lambda_j d(Tv_j,Tv'_j)^q\bigg)^{1/q} \le C  \n{ \bigg(\sum_{j=1}^n \lambda_j|v_j-v'_j|^q\bigg)^{1/q} }_E.
$$
The smallest such constant $C$ is the \emph{Lipschitz $p$-concavity constant} of $T$ and is denoted by $M_{(p)}^{\Lip}(T)$.
\end{definition}

This time we can apply the idea from \cite{Farmer-Johnson-09} and note that it suffices to check the condition with all the $\lambda_j$ being equal to one. From that it one can deduce the same inequality for integer values of $\lambda_j$, hence rational values, and for arbitrary values of $\lambda_j$ we approximate with increasing sequences of rationals.
We will primarily be interested in the case when $X$ is a Banach space.
Note that when $X$ is a Banach space and $T : E \to X$ is linear, clearly $T$ is $p$-concave if and only if it is Lipschitz $p$-concave (and with the same constant).

Unlike in the Lipschitz $p$-convexity situation, there is no natural candidate for a linearized map that is $q$-concave if and only if $T$ is Lipschitz $q$-concave, because now the metric space is not the domain and thus there is no canonical linearization.

The following factorization theorem and its proof are inspired by the corresponding result for linear maps proven by Krivine \cite{Krivine-factorisation-reticules}.
The presentation follows that of \cite[Thm. 1.d.11]{Lindenstrauss-Tzafriri-II}.

\begin{theorem}\label{theorem-lipschitz-concavity}
Let $X$, $Y$ be metric spaces with $Y$ complete and $E$ a Banach lattice.
Suppose that $T : X \to E$ is Lipschitz $p$-convex and $S : E \to Y$ is Lipschitz $p$-concave.
Then the map $ST$ can be factored through an $L_p(\mu)$ space.
Moreover, we may arrange to have $ST = S_1T_1$ with $T_1 : X \to L_p(\mu)$, $S_1 : L_p(\mu)\to Y$, $\Lip(T_1) \le M^{(p)}_{\Lip}(T)$ and $\Lip(S_1) \le M^{\Lip}_{(p)}(S)$.
\end{theorem}

\begin{proof}
Let $I_T$ be the (in general non-closed) ideal of $E$ generated by the range of $T$.
We define new operations on $I_T$ as in the usual $p$-concavification procedure, that is, for $x, y \in I_T$ and real $\alpha$ put
$$
x \oplus y := (x^p+y^p)^{1/p}, \qquad \qquad \alpha \odot x := \alpha^{1/p} x,
$$
and let $\check{I}_T$ denote the vector lattice obtained when $I_T$ is endowed with the original order and the operations $\oplus, \odot$. Set
\begin{multline*}
 F_1 := \conv \Big\{ x \in \check{I}_T : |x| \le \lambda|Tv-Tv'| \\ \text{ for some } v,v' \in X, \lambda>0  \text{ with } \lambda d(v,v') < 1/M^{(p)}_{\Lip}(T) \Big\} 
\end{multline*}
and
\begin{multline*}
 F_2 := \conv \Big\{ x \in \check{I}_T : x>0 \text{ and } \eta d(Sy,Sy') \ge M^{\Lip}_{(p)}(S) \\ \text{ for some } y,y' \in L, \eta>0 \text{ with } \eta|y-y'| \le x \Big\}.
\end{multline*}
where both convex hulls are taken in the sense of $\check{I}_T$, i.e. using the operations $\oplus, \odot$.

If $x$ belongs to $F_1$, then it can be written the form $\bigoplus_j \alpha_j \odot x_j$ where $\alpha_j \ge 0$, $\sum_j \alpha_j =1$ and $|x_j| \le \lambda_j |Tv_j -Tv'_j|$ with $\lambda_j d(v_j,v'_j) < 1/ M^{(p)}_{\Lip}(T)$. Therefore,
\begin{align*}
\n{x} &= \bigg\| \Big( \sum_j |\alpha_j^{1/p} x_j|^p  \Big)^{1/p} \bigg\|  \le \bigg\| \Big( \sum_j \alpha_j \lambda_j^p |Tv_j-Tv'_j|^p  \Big)^{1/p} \bigg\| \\
&\le M^{(p)}_{\Lip}(T) \Big( \sum_j \alpha_j \lambda_j^p d(v_j,v'_j)^p \Big)^{1/p} < 1.
\end{align*}

On the other hand, if $x$ belongs to $F_2$ then it can be written as $\bigoplus_j \beta_j \odot x_j$ where $\beta_j \ge 0$, $\sum_j \beta_j =1$ and $x_j \ge \eta_j |y_j-y'_j|$ with $\eta_j \ge 0$ and $\eta_j d(Sy_j,Sy'_j) \ge M^{\Lip}_{(p)}(S)$.
Therefore
\begin{align*}
\n{x} &= \bigg\| \Big( \sum_j |\beta_j^{1/p} x_j|^p  \Big)^{1/p} \bigg\| \ge \bigg\| \Big( \sum_j \beta_j \eta_j^p|y_j-y'_j|^p  \Big)^{1/p} \bigg\| \\
&\ge \frac{1}{M^{\Lip}_{(p)}(S)}  \Big( \sum_j \beta_j \eta_j^p d(Sy_j,Sy'_j)^p \Big) \ge 1.
\end{align*}

Hence, $F_1 \cap F_2 = \emptyset$ and since $0$ is an internal point of $F_1$
it follows from the separation theorem that there exists a linear functional $\varphi$ on $\check{I}_T$ such that $\varphi(x) \le 1$ for all $x \in F_1$ and $\varphi(x) \ge 1$ for all $x \in F_2$.
Note that from the definition of $F_2$, for any positive real $\alpha$, any positive $x$ in $\check{I}_T$ and any $x_0 \in F_2$ we have that $\alpha \odot x \oplus x_0$ belongs to $F_2$.
It follows that $\varphi(x) \ge 0$ whenever $0 < x \in \check{I}_T$ and, thus, we can define a seminorm on $I_T$ by putting
$$
\n{x}_0 := \varphi(|x|)^{1/p}, \qquad x \in I_T.
$$
Let $\alpha$ be a real number and $x \in I_T$. Then
$$
\n{\alpha x}_0 = \varphi(|\alpha||x|)^{1/p} = \varphi(|\alpha|^p \odot |x|)^{1/p} = \big[|\alpha|^p \varphi(|x|)  \big]^{1/p} = |\alpha| \n{x}_0.
$$
Let $x,y \in I_T$. Note that $|x|+|y| = \big(|x|^{1/p}\oplus|y|^{1/p}\big)^{p}$.
On the other hand, from the lattice functional calculus and H\"older's inequality, whenever $\alpha$ and $\beta$ are positive reals with $\alpha^{p'}+\beta^{p'}=1$ we have
$$
\big(|x|^{1/p}\oplus|y|^{1/p}\big)^{p} \le \alpha^{-p} \odot |x| + \beta^{-p} \odot |y|.
$$
Hence
\begin{align*}
\n{x+y}_0^p &= \varphi\big(|x+y|\big) \le \varphi\big(|x|+|y|\big) = \varphi\big( \big(|x|^{1/p}\oplus|y|^{1/p}\big)^{p} \big) \\
&\le  \varphi\big( \alpha^{-p} \odot |x| + \beta^{-p} \odot |y| \big) = \alpha^{-p}\varphi(|x|) + \beta^{-p}\varphi(|y|) \\
&= \alpha^{-p} \n{x}_0^p + \beta^{-p}\n{y}_0^p.
\end{align*}
Therefore, setting
$$
\alpha := \frac{\n{x}_0^{1/p'}}{(\n{x}_0+\n{y}_0)^{1/p'}} \qquad\text{and}\qquad \beta := \frac{\n{y}_0^{1/p'}}{(\n{x}_0+\n{y}_0)^{1/p'}}
$$
we satisfy the condition $\alpha^{p'}+\beta^{p'}=1$, while
$$
\alpha^{-p} = \frac{(\n{x}_0+\n{y}_0)^{p-1}}{\n{x}_0^{p-1}} \qquad\text{and}\qquad  \beta^{-p} = \frac{(\n{x}_0+\n{y}_0)^{p-1}}{\n{y}_0^{p-1}}
$$
so we conclude
$$
\n{x+y}_0^p \le \frac{(\n{x}_0+\n{y}_0)^{p-1}}{\n{x}_0^{p-1}}\n{x}_0^p + \frac{(\n{x}_0+\n{y}_0)^{p-1}}{\n{y}_0^{p-1}}\n{y}_0^p  = \big( \n{x}_0 + \n{y}_0 \big)^p,
$$
and thus $\n{x+y}_0 \le \n{x}_0 + \n{y}_0$.

Observe now that for any $x,y \in I_T$ we have
$$
|x|+|y| \ge \big( |x|^p + |y|^p \big)^{1/p} \ge |x| \vee |y|
$$
since these inequalities are valid for reals. By the fact that $\varphi$ is non-negative, we get that
\begin{align*}
\big\| |x|+|y| \big\|_0^p &= \varphi\big(|x|+|y|\big) \ge \varphi\big((|x|^p+|y|^p)^{1/p}\big) = \varphi(|x|\oplus|y|) = \varphi(|x|) + \varphi(|y|) \\
&= \n{x}_0^p + \n{y}_0^p \ge \varphi\big( |x|\vee|y| \big) = \big\| |x|\vee|y| \big\|_0^p. 
\end{align*}
This inequality concerning $\n{\cdot}_0$ clearly remains valid in the completion $Z$ of $I_T$ modulo the ideal of all $x \in I_T$ for which $\n{x}_0 = 0$.
Therefore, if $|x| \wedge |y| = 0$ for some $x$ and $y$ in the lattice $Z$ then (recalling that $|x|\wedge|y| = |x|+|y|-|x|\vee|y|$) we obtain
$$
\big\| |x|+|y| \big\|_0^p = \n{x}_0^p + \n{y}_0^p,
$$
i.e. $Z$ is an abstract $L_p$ space. It follows from the $L_p$ version of Kakutani's representation theorem that $Z$ is order isometric to an $L_p(\mu)$ space for a suitable measure $\mu$.

Let $T_1 : X \to Z$ be defined by $T_1v = Tv$, $v \in X$, i.e. the same as $T$ but considered as an map into $Z$.
For $v,v' \in X$ and $\lambda>0$, if $\lambda d(v,v') < 1/M^{(p)}_{\Lip}(T)$ then $\lambda(T_1v - T_1v') \in F_1$, which implies that $\varphi(\lambda(T_1v-T_1v')) \le 1$ and thus $\n{\lambda(T_1v-T_1v')}_0 \le 1$, from where it follows that
$\n{T_1v-T_1v'}_0 \le M^{(p)}_{\Lip}(T) d(v,v')$, i.e. $\Lip(T_1) \le M^{(p)}_{\Lip}(T)$.

Let $S_1 : I_T/\ker(\n{\cdot}_0) \to Y$ be defined by $S_1x = Sx$, $x \in I_T$.
Note that this is well defined: if $Sx\not= Sx'$, then $\tfrac{M^{\Lip}_{(p)}(S)}{d(Sx,Sx')}|x-x'|$ belongs to $F_2$, so $\varphi\Big(\tfrac{M^{\Lip}_{(p)}(S)}{d(Sx,Sx')}|x-x'|\Big) \ge 1$ and in particular $\n{x-x'}_0\not= 0$.
By an argument similar to the one for $T_1$, this defines a Lipschitz map from $I_T/\ker(\n{\cdot}_0)$ to $Y$ with Lipschitz constant at most $M^{\Lip}_{(p)}(S)$.
Since $I_T/\ker(\n{\cdot}_0)$ is dense in $Z$ and $Y$ is complete, this can be extended to a Lipschitz map $S_1 : Z \to Y$ with the same Lipschitz constant, giving the desired factorization.
\end{proof}

\section{Factorization theorems}\label{sec-factorization-theorems}

In this section we prove some results that make very clear why our Lipschitz versions of the  Maurey/Nikishin and Krivine  factorization Theorems hold.
They are nonlinear generalizations of the beautiful factorization theorems of Raynaud and Tradacete \cite{Raynaud-Tradacete}, that characterize $p$-convex and $q$-concave linear maps in terms of factorizations through $p$-convex and $q$-concave Banach lattices.

\subsection{Factorization of Lipschitz $p$-convex maps}

We start by proving the following characterization of Lipschitz $p$-convex maps, compare to \cite[Thm. 3]{Raynaud-Tradacete}.

	\begin{theorem}\label{thm-factorization-lipschitz-p-convex}
		Let $E$ be a Banach lattice, $X$ a metric space and $1 \le p \le \infty$. A
		Lipschitz map $T : X \to E$ is Lipschitz $p$-convex if and only if there exist a $p$-convex Banach
		lattice $W$, a positive operator (in fact, an injective interval preserving lattice
		homomorphism) $\psi : W \to E$ and another Lipschitz map $R : E \to W$
such that $T = \psi \circ R$.
	$$
	\xymatrix{
	X \ar[rr]^{T}\ar[dr]_{R} & &E  \\
	 &W \ar[ru]_{\psi} & 
	}
	$$
	Moreover, $M^{(p)}_{\Lip}(T) = \inf \Lip(R) \cdot M^{(p)}(I_W) \cdot \n{\psi}$ where the infimum is taken over all such factorizations.
	\end{theorem}

\begin{remark}
	Theorem \ref{thm-factorization-lipschitz-p-convex} in fact follows as an immediate corollary from Theorem \ref{theorem-lipschitz-convexity} and \cite[Thm. 3]{Raynaud-Tradacete}, but we have chosen to show a direct proof.
In this way, Theorem \ref{thm-factorization-lipschitz-p-convex} (together with \cite[Thm. 3]{Raynaud-Tradacete}) can be used to provide an alternative proof of Theorem \ref{theorem-lipschitz-convexity} that does not rely on duality arguments.
\end{remark}

\begin{proof}[Proof of Theorem \ref{thm-factorization-lipschitz-p-convex}]
We will assume that $1\le p<\infty$, the proof of the case $p=\infty$ can be obtained via the usual changes.

If we do have such a factorization,
consider $x_j,x'_j \in X$. By the positivity of $\psi$, using \cite[Prop. 1.d.9]{Lindenstrauss-Tzafriri-II} and the $q$-convexity of $W$,
\begin{multline*}
	\n{ \bigg(\sum_{j=1}^n |Tx_j-Tx'_j|^p\bigg)^{1/p} }_E
	= \n{ \bigg(\sum_{j=1}^n |\psi (Rx_j-Rx'_j)|^p\bigg)^{1/p} }_E \\
\le \n{\psi} \n{ \bigg(\sum_{j=1}^n |\psi (Rx_j-Rx'_j)|^p\bigg)^{1/p} }_W \\
\le \n{\psi} M^{(p)}(I_W) \bigg(\sum_{j=1}^n  \n{Rx_j - Rx'_j}_W^p\bigg)^{1/p} \\
\le \n{\psi} M^{(p)}(I_W) \Lip(R) \bigg(\sum_{j=1}^n  d(x_j,x'_j)^p\bigg)^{1/p},
\end{multline*}
showing that $T$ is Lipschitz $p$-convex with
$M^{(p)}_{\Lip}(T) \le \n{\psi} M^{(p)}(I_W) \Lip(R)$.
\\
\\
Now let $T : X \to E$ be Lipschitz $p$-convex.
Consider the set
\begin{multline*}
A := \bigg\{ u \in E \; : \; |u| \le \bigg( \sum_{i=1}^k \lambda_i|Tx_i-Tx_i'|^p \bigg)^{1/p}, \\
\text{ where } \lambda_i\ge 0 \text{ and } \sum_{i=1}^k\lambda_id(x_i,x_i')^p \le 1 \bigg\}.	
\end{multline*}
We can consider the Minkowski functional defined by its closure $\bar{A}$ in $E$,
$$
\n{z}_W = \inf\{ \mu>0 \; : \; z \in \mu \bar{A} \}.
$$
Clearly $\bar{A}$ is solid, and since $T$ is Lipschitz $p$-convex, it is also a bounded subset of $E$.
Let us consider the space $W = \{ z \in L \; : \; \n{z}_W < \infty \}$.
We claim that for any $z_1 , \dotsc, z_n$ in $W$ it follows that $\Big(\sum_{i=1}^n|z_i|^p\Big)^{1/p}$ is in $W$ and moreover
$$
\n{\bigg(\sum_{i=1}^n|z_i|^p\bigg)^{1/p}}_W \le \bigg(\sum_{i=1}^n\n{z_i}_W^p\bigg)^{1/p}.
$$
Given $\varepsilon>0$, for each $i = 1, \dotsc, n$ there exist $\mu_i$ with $z_i \in \mu_i\bar{A}$ such that $\mu_i^p \le \n{z_i}_W + \frac{\varepsilon^p}{n}$.
Thus, for every $\delta>0$ there exist $\{z_i^\delta\}_{i=1}^n$ in $E$ such that $\n{z_i-z_i^\delta}_E \le \delta$ and 
$$
|z_i^\delta| \le \bigg( \sum_{j=1}^{m_{i,\delta}} \lambda^{i,\delta}_j|Tx^{i,\delta}_j-Ty^{i,\delta}_j|^p \bigg)^{1/p}
$$
where the nonnegative numbers $\{ \lambda^{i,\delta}_j \}_{j=1}^{m_{i,\delta}}$    and the points $\{x^{i,\delta}_j, y^{i,\delta}_j\}_{j=1}^{m_{i,\delta}}$ in $X$ satisfy
$$
\bigg( \sum_{j=1}^{m_{i,\delta}} \lambda^{i,\delta}_j d(x^{i,\delta}_j,y^{i,\delta}_j)^p \bigg)^{1/p} \le \mu_i
$$
for each $i = 1, \dotsc, n$ and each $\delta>0$.
For each $\delta>0$, define
$$
w_\delta = \bigg(\sum_{i=1}^n|z^\delta_i|^p\bigg)^{1/p}.
$$
Then
\begin{multline*}
	\n{ \bigg(\sum_{i=1}^n|z_i|^p\bigg)^{1/p} - w_\delta }_E \le 
	\n{ \bigg(\sum_{i=1}^n|z_i - z_i^\delta|^p\bigg)^{1/p} }_E  \\
	\le \bigg(\sum_{i=1}^n \n{z_i - z_i^\delta}_E^p\bigg)^{1/p} \le n\delta.	
\end{multline*}
Moreover, we will show that for every $\delta>0$, $w_\delta$ belongs to $\big(\sum_{i=1}^n \mu_i^p\big)^{1/p} A$.
Indeed,
$$
|w_\delta| = \bigg(\sum_{i=1}^n|z^\delta_i|^p\bigg)^{1/p} \le
\bigg(\sum_{i=1}^n \sum_{j=1}^{m_{i,\delta}} \lambda^{i,\delta}_j|Tx^{i,\delta}_j-Ty^{i,\delta}_j|^p \bigg)^{1/p}
$$
and
$$
\bigg(\sum_{i=1}^n \sum_{j=1}^{m_{i,\delta}} \lambda^{i,\delta}_j d(x^{i,\delta}_j,y^{i,\delta}_j)^p \bigg)^{1/p} \le \bigg(\sum_{i=1}^n\mu_i^p\bigg)^{1/p}. 
$$
Hence, $\big(\sum_{i=1}^n|z_i|^p\big)^{1/p}$ belongs to $\big(\sum_{i=1}^n\mu_i^p\big)^{1/p}\bar{A}$. Therefore, from the definition of $\n{\cdot}_W$ and the choice of $\mu_i$,
\begin{multline*}
	\n{\bigg(\sum_{i=1}^n|z_i|^p\bigg)^{1/p}}_W \le \bigg(\sum_{i=1}^n\mu_i^p\bigg)^{1/p} \\ \le \bigg(\sum_{i=1}^n (\n{z_i}_W^p +\varepsilon^p/n ) \bigg)^{1/p}
	\le  \bigg(\sum_{i=1}^n \n{z_i}_W^p \bigg)^{1/p} + \varepsilon.
\end{multline*}
Letting $\varepsilon$ go to 0, we conclude
$$
\n{\bigg(\sum_{i=1}^n|z_i|^p\bigg)^{1/p}}_W \le \bigg(\sum_{i=1}^n \n{z_i}_W^p \bigg)^{1/p}.
$$
It follows that the Minkowski functional $\n{\cdot}_W$ is in fact a norm on $W$.
Since $\bar{A}$ is bounded, $\n{z}_W = 0$ implies that $z = 0$.
Moreover if $u,v \in W$ are not zero, set
$$
\tilde{u} = \frac{|u|}{\n{u}_W}, \quad, \tilde{v} = \frac{|v|}{\n{v}_W},
\alpha = \frac{\n{u}_W}{\n{u}_W + \n{v}_W}, \quad \beta = \frac{\n{v}_W}{\n{u}_W + \n{v}_W}.
$$
Since $\n{\tilde{u}}_W = \n{\tilde{v}}_W = 1$, $\alpha,\beta \ge 0$ and $\alpha + \beta = 1$ we have
\begin{align*}
	\n{u+v}_W &\le \n{ |u|+|v| }_W 
	= \big( \n{u}_W + \n{v}_W \big) \n{\alpha \tilde{u} + \beta \tilde{v}}_W \\
	&\le \big( \n{u}_W + \n{v}_W \big) \n{(\alpha \tilde{u}^p + \beta \tilde{v}^p)^{1/p}}_W \\
	&\le \big( \n{u}_W + \n{v}_W \big) \big( \alpha \n{\tilde{u}}_W^p + \beta \n{\tilde{v}}_W^p \big) \\
	&= \n{u}_W + \n{v}_W.	
\end{align*}
Thus, $(W, \n{\cdot}_W)$ is a $p$-convex normed lattice with constant 1.
Let $(w_i)_{i=1}^\infty$ be a Cauchy sequence in $W$.
Since for every $z \in E$ it holds that $\n{z}_E \le M_{\Lip}^{(p)}(T) \n{z}_W$,
it follows that  $(w_i)_{i=1}^\infty$ is also a Cauchy sequence in $E$ and this has a limit $w$ in $E$.
Notice that since the $w_i$ are bounded in $W$, there exists some finite $\mu$ such that $w_i \in \mu \bar{S}$ for every $i \in \N$ and since $\bar{S}$ is closed in $E$, we must have $w \in \mu \bar{S}$.
Thus, $w$ belongs to $W$, and we will show that $(w_i)_{i=1}^\infty$ converges to $w$ also in $W$.
To this end, let $\varepsilon>0$. Since $(w_i)_{i=1}^\infty$ is a Cauchy sequence, there exists $N\in\N$ such that $w_i-w_j \in \varepsilon \bar{S}$ whenever $i, j \ge N$.
Thus, if $i \ge N$ we can write
$$
w - w_i = (w - w_j) + (w_j - w_i) 
$$
for every $j \ge N$, and letting $j \to \infty$ we obtain that $w - w_i \in \varepsilon\bar{S}$.
This shows that $(w_i)_{i=1}^\infty$ converges to $w$ in $W$, hence $W$ is complete and therefore a Banach lattice.

Let us observe that from the definition of $W$ we clearly have $\n{Tx-Tx'}_W \le d(x,x')$ for every $x,x' \in X$, so the map $R : X \to W$ given by $Rx = Tx$ is Lipschitz with $\Lip(R) \le 1$. 
Moreover, as noticed above it also holds that $\n{z}_E \le M_{\Lip}^{(p)}(T) \n{z}_W$ for each $z \in E$, hence the formal inclusion $\psi : W \to E$ is clearly and injective interval preserving lattice homomorphism with norm at most $M_{\Lip}^{(p)}(T)$, and $T = \psi \circ R$.
\end{proof}

Similarly to the notation in \cite{Raynaud-Tradacete}, since the lattice $W$ constructed in the proof of Theorem \ref{thm-factorization-lipschitz-p-convex} depends on $T$ and $p$ we will denote it by $W_{T,p}$.
As in \cite[Remark 4]{Raynaud-Tradacete}, note that whenever $1 \le p' \le p$ and $T : X \to E$ is Lipschitz $p$-convex, it always holds that the inclusion $W_{T,p'} \to W_{T,p}$ has norm at most one. 
Moreover, as in \cite[Remark 5]{Raynaud-Tradacete}, note that when $E$ is a $p$-convex Banach lattice and $T$ is the identity map on $E$, then $W_{T,p}$ is a renorming of $E$ with $p$-convexity constant equal to one.
The Banach lattice $W_{T,p}$ also satisfies a certain minimality property that follows from \cite[Prop. 4]{Raynaud-Tradacete}, see that paper for the details.

\subsection{Factorization of Lipschitz $q$-concave maps}

The following characterization of Lipschitz $q$-concave maps is a nonlinear generalization of \cite[Thm. 1]{Raynaud-Tradacete}.

\begin{theorem}\label{thm-factorization-lipschitz-q-concave}
	Let $E$ be a Banach lattice, $X$ a complete metric space and $1 \le q \le \infty$. A
	Lipschitz map $T : E \to X$ is Lipschitz $q$-concave if and only if there exist a $q$-concave Banach
	lattice $V$, a positive operator $\phi : E \to V$ (in fact, a lattice homomorphism with dense image), and another Lipschitz map $S : V \to X$ such that $T = S \circ \phi$.
	$$
	\xymatrix{
	E \ar[rr]^{T}\ar[dr]_{\phi} & &X  \\
	 &V \ar[ru]_{S} & 
	}
	$$
	Moreover, $M_{(q)}^{\Lip}(T) = \inf \n{\phi} \cdot M_{(q)}(I_V) \cdot \Lip(S)$.
\end{theorem}

\begin{remark}
	The proof in our case is inspired by the linear one, but as far as we can tell the Lipschitz case does not follow from the linear one.
	Unlike in the case of Theorem \ref{theorem-lipschitz-convexity}, it does not seem to be possible to fully reduce Lipschitz $q$-concave maps to (linear) $q$-concave maps.
	However, comparing Theorem \ref{thm-factorization-lipschitz-q-concave} with  \cite[Thm. 1]{Raynaud-Tradacete} shows that Lipschitz $q$-concave maps are in fact (linear) $q$-concave maps followed by Lipschitz maps.
\end{remark}

\begin{proof}[Proof if Theorem \ref{thm-factorization-lipschitz-q-concave}]
	The case $q=\infty$ is trivial because every Banach lattice is $\infty$-concave, so let us assume that $1 \le q < \infty$.
	The specific construction carried out below, however, has an analogue in the case $q=\infty$.
	
	First, let us assume that such a factorization exists. 
	Consider $u_j,u'_j \in E$. By the positivity of $\phi$, using \cite[Prop. 1.d.9]{Lindenstrauss-Tzafriri-II} and the $q$-concavity of $V$,
	\begin{multline*}
		\bigg(\sum_{j=1}^n  d(Tu_j,Tu'_j)^q\bigg)^{1/q}
		= \bigg(\sum_{j=1}^n  d(S\phi u_j,S\phi u'_j)^q\bigg)^{1/q} \\
		\le \Lip(S) \bigg(\sum_{j=1}^n  \n{\phi u_j - \phi u'_j}_V^q\bigg)^{1/q} \\
		\le \Lip(S) M_{(q)}(I_V) \n{ \bigg(\sum_{j=1}^n |\phi(v_j-v'_j)|^q\bigg)^{1/q} }_V \\
		\le \Lip(S) M_{(q)}(I_V) \n{\phi} \n{ \bigg(\sum_{j=1}^n |v_j-v'_j|^q\bigg)^{1/q} }_E, 
	\end{multline*}
	showing that $T$ is Lipschitz $q$-concave with
	$M_{(q)}^{\Lip}(T) \le \n{\phi} \cdot M_{(q)}(I_V) \cdot \Lip(S)$.
	
	Now suppose that $T$ is Lipschitz $q$-concave.
	Given $u \in E$, define
	$$
	\rho(u) = \sup \left\{ \bigg(\sum_{j=1}^n  \lambda_j d(Tu_j,Tu'_j)^q\bigg)^{1/q} 
	\; : \; \lambda_j\ge0, \; \bigg(\sum_{j=1}^n \lambda_j |u_j-u'_j|^q\bigg)^{1/q} \le |u| \right\}.
	$$
	Note that since
	$$
	\bigg(\sum_{j=1}^n  \lambda_j d(Tu_j,Tu'_j)^q\bigg)^{1/q} \le M_{(q)}^{\Lip}(T)  \n{\bigg(\sum_{j=1}^n \lambda_j |u_j-u'_j|^q\bigg)^{1/q} }_E,
	$$
	we have for all $u,v \in E$
	$$
	d\big(Tu,Tv\big) \le \rho(u-v) \le M_{(q)}^{\Lip}(T) \n{u-v}.
	$$
	Moreover, we claim that $\rho$ is a lattice seminorm on $E$.
	First notice that for any $u,v \in E$ and $\lambda \ge 0$, it is clear that $\rho(\lambda u) = \lambda \rho(u)$ and that $|v| \le |u|$ implies $\rho(v) \le \rho(u)$.
	To prove the triangle inequality, let $u, v \in E$ and $w = |u| + |v|$, and let $I_w \subset E$ be the ideal generated by $w$ in $E$, which is identified with a space $C(K)$ in which $w$ corresponds to the function identically equal to 1 \cite[II.7]{Schaefer}.
	Given $\varepsilon>0$, find $w_1, w_1', \dotsc, w_n,w_n' \in E$ and $\lambda_1, \dotsc, \lambda_n \ge 0$ such that
	$$
	\bigg(\sum_{j=1}^n \lambda_j |w_j-w'_j|^q\bigg)^{1/q} \le |w| \quad\text{and}\quad
	\rho(w) \le \bigg(\sum_{j=1}^n  \lambda_j d(Tw_j,Tw'_j)^q\bigg)^{1/q} + \varepsilon.
	$$
	Since $u,v \in I_w$, they correspond to functions $f, g \in C(K)$ such that $|f(t)| + |g(t)| = 1$ for all $t \in K$. 
	Similarly, each $w_j-w_j'$ corresponds to $h_j \in C(K)$ with $(\sum_{j=1}^n \lambda_j|h_j(t)|^q)^{1/q} \le 1$ for all $t \in K$.
	Let us now consider
	$$
	f_j(t) = h_j(t) f(t), \qquad g_j(t) = h_j(t) g(t)
	$$
	which belong to $C(K)$ and satisfy
	$$
	\bigg(\sum_{j=1}^n \lambda_j |f_j(t)|^q\bigg)^{1/q} \le |f(t)|
	\quad\text{and}\quad
	\bigg(\sum_{j=1}^n \lambda_j |g_j(t)|^q\bigg)^{1/q} \le |g(t)|.
	$$
	This means there are $u_j, v_j \in I_w \subset E$ with $u_j + v_j = w_j-w_j'$ for each $1\le j \le n$ and satisfying
	$$
	\bigg(\sum_{j=1}^n \lambda_j |u_j|^q\bigg)^{1/q} \le |u|
	\quad\text{and}\quad
	\bigg(\sum_{j=1}^n \lambda_j |v_j|^q\bigg)^{1/q} \le |v|.
	$$
	Notice that $w_j-u_j = v_j+w_j'$, and hence
	$$
		d(Tw_j,Tw'_j) \le d\big(Tw_j, T(w_j-u_j)\big) + d\big(T(v_j+w'_j),Tw'_j\big).
	$$
	Then,
	\begin{multline*}
		\rho(u+v) \le \rho(w) \le \bigg(\sum_{j=1}^n  \lambda_j d(Tw_j,Tw'_j)^q\bigg)^{1/q} + \varepsilon \\
		\le \bigg(\sum_{j=1}^n  \lambda_j d\big(Tw_j, T(w_j-u_j)\big)^q\bigg)^{1/q}
		 + \bigg(\sum_{j=1}^n  \lambda_j d\big(T(v_j+w'_j),Tw'_j\big)^q\bigg)^{1/q} + \varepsilon.
	\end{multline*}
	Since $w_j - (w_j-u_j) = u_j$ and $(v_j+w'_j) - w'_j = v_j$, it follows from the definition of $\rho$ that
	\begin{multline*}
		\bigg(\sum_{j=1}^n  \lambda_j d\big(Tw_j, T(w_j-u_j)\big)^q\bigg)^{1/q} \le \rho(u) \qquad \text{and} \\
		\bigg(\sum_{j=1}^n  \lambda_j d\big(T(v_j+w'_j),Tw'_j\big)^q\bigg)^{1/q} \le \rho(v)	
	\end{multline*}
		
	and thus $\rho(u+v) \le \rho(u) + \rho(v) + \varepsilon$.
	Letting $\varepsilon$ go to 0, we have proved that $\rho$ satisfies the triangle inequality.
	
	Let $V$ denote the Banach lattice obtained by completing $E/\rho^{-1}(0)$ with the norm induced by $\rho$, and let $\varphi$ be the quotient map $E \to E/\rho^{-1}(0)$ seen as a map to $V$.
	For $u \in E$ let us define $S(\phi(u)) = T(u)$.
	Since $d\big(Tu,Tv\big) \le \rho(u-v)$ for any $u, v \in E$ the map $S$ is well defined on $E/\rho^{-1}(0)$. Moreover, since $X$ is complete we can extend $S$ to a Lipschitz map $S : V \to X$ such that $\Lip(S) \le 1$ and $T = S \varphi$.
	
	Now consider $\{u_i\}_{i=1}^n$ in $E$, and let $\varepsilon>0$.
	For each $i =1 , \dotsc, n$ there exist $\{v_j^i,w_j^i\}_{j=1}^{k_i}$ in $E$ and nonnegative numbers $\{\lambda_j^i\}_{j=1}^{k_i}$ such that
\begin{multline*}
	\bigg(\sum_{j=1}^{k_i} \lambda_j^i |v_j^i-w^i_j|^q\bigg)^{1/q} \le |u_i|  \qquad \text{and} \\
		\rho(u_i)^q \le \sum_{j=1}^{k_i}  \lambda_j^i d(Tv_j^i,Tw_j^i)^q + \frac{\varepsilon^q}{n}.
\end{multline*}
Adding up, we have that
	$$
	\bigg(\sum_{i=1}^{n} \rho(u_i)^q\bigg)^{1/q} \le 
	\rho\bigg( \bigg( \sum_{i=1}^{n} |u_i|^q\bigg)^{1/q} \bigg) + \varepsilon.
	$$
	Letting $\varepsilon$ go to 0 we conclude that the normed lattice $E/\rho^{-1}(0)$ is $q$-concave with constant 1, and thus so is its completion $V$.
	
	Finally, note that
	$$
	\n{\phi} \cdot M_{(q)}(I_V) \cdot \Lip(S) \le M_{(q)}^{\Lip}(T) \cdot 1 \cdot 1 = M_{(q)}^{\Lip}(T).
	$$ 
\end{proof}

Since the lattice $V$ constructed in the previous proof depends on the map $T : E \to X$ and $q$, we will denote it by $V_{T,q}$ whenever needed.
As pointed out in the proof, $V_{T,q}$ has $q$-concavity constant one. In particular if $E$ is $q$-concave and $T$ is the identity on $E$, then $V_{T,q}$ is the usual lattice renorming of E with $q$-concavity constant one.

The factorization given by the proof of Theorem \ref{theorem-lipschitz-concavity} is maximal in a certain sense, as the following proposition shows. We omit the proof, since it is an easy combination of that of \cite[Prop. 2]{Raynaud-Tradacete} and the ideas we used to prove Theorem \ref{thm-factorization-lipschitz-q-concave}.

\begin{proposition}
	Let $E$ be a Banach lattice, $X$ a complete metric space, $1 \le q \le \infty$ and
	 $T : E \to X$ be a Lipschitz $q$-concave map. Suppose that $T$ factors through a $q$-concave Banach lattice $\tilde{V}$ with factors $A : E \to \tilde{V}$ and $B : \tilde{V} \to X$, such that $A$ is a lattice homomorphism whose image is dense in $\tilde{V}$, $B$ is a Lipschitz map, and $T = B \circ A$. 
Then there is a lattice homomorphism $\varphi : \tilde{V} \to V_{T,q}$ such that
$\phi = \varphi \circ A$ and $S \circ \varphi = B$.

$$
\xymatrix{
E \ar[rr]^{T}\ar[ddr]_{\phi} \ar[dr]^{A} & &X  \\
 &\tilde{V} \ar@{-->}[d]^{\varphi}
 \ar[ru]^{B} & \\
 &V_{T,q} \ar[ruu]_{S}& 
}
$$
\end{proposition}

\begin{remark}
	In \cite[Sec. 3]{Raynaud-Tradacete}, duality relations are established between the lattices constructed in the factorization theorems for $p$-convex and $q$-concave operators. We do not try to do something similar here, since the Lipschitz dual $L^\#$ of a Banach lattice is not necessarily a Banach lattice.
\end{remark}

\subsection{Factorization theorems \`a la Krivine}\label{sec-factorization-theorems-a-la-Krivine}

Once again taking \cite{Raynaud-Tradacete} as a model, the nonlinear Krivine factorization theorem (Theorem \ref{theorem-lipschitz-concavity}) follows easily from the factorization Theorems \ref{thm-factorization-lipschitz-p-convex} and \ref{thm-factorization-lipschitz-q-concave}.

\begin{proof}[Alternative proof of Theorem \ref{theorem-lipschitz-concavity}]
Let $X$, $Y$ be metric spaces with $Y$ complete and $E$ a Banach lattice.
Suppose that $T : X \to E$ is Lipschitz $p$-convex and $S : E \to Y$ is Lipschitz $p$-concave. 
Apply Theorems \ref{thm-factorization-lipschitz-p-convex} and \ref{thm-factorization-lipschitz-q-concave} (more specifically, their proofs) to obtain factorizations
$$
\xymatrix{
X \ar[rr]^{T}\ar[dr]_{\tau} & &E \ar[rr]^{S}\ar[dr]_{\phi} & &Y  \\
 &W \ar[ru]_{\psi}  &  &V \ar[ru]_{\sigma} & \\
}
$$
where $W$ (resp. $V$) is a $p$-convex (resp. $p$-concave) Banach lattice with constant 1, 
$\tau$ and $\sigma$ are Lipschitz maps with constant at most $1$, and
$\psi$ (resp. $\phi$) is a lattice
homomorphism of norm at most $M_{\Lip}^{(p)}(T)$ (resp. $M^{\Lip}_{(p)}(S)$ ).
From \cite[Lemma 17]{Raynaud-Tradacete}, $\phi \circ \psi$ factors through an $L_p(\mu)$ space with the factors being lattice homomorphisms whose norms have product at most
 $M_{\Lip}^{(p)}(T) \cdot M^{\Lip}_{(p)}(S)$.
The conclusion of Theorem \ref{theorem-lipschitz-concavity} is now clear.
\end{proof}

Similarly, Theorems \ref{thm-factorization-lipschitz-p-convex} and \ref{thm-factorization-lipschitz-q-concave} allow us to reduce the following two Lipschitz results to their linear counterparts (namely Remark 10, Prop. 18 and Cor. 19 in \cite{Raynaud-Tradacete}).

\begin{proposition}\label{prop-generalized-krivine}
	Let $X$, $Y$ be metric spaces with $Y$ complete and $E$ a Banach lattice.
	Suppose that $T : X \to E$ is Lipschitz $p$-convex and $S : E \to Y$ is Lipschitz $q$-concave, with $1 \le q < p < \infty$. Then $ST$ factors through a canonical inclusion $i_{p,q} : L_p(\mu) \to L_q(\mu)$. In fact, there is a factorization
	$$
	\xymatrix{
	X \ar[r]^{T}\ar[d]_{T_1} &E \ar[r]^{S} &Y  \\
	L_p(\mu) \ar[rr]^{i_{p,q}} &   & L_q(\mu) \ar[u]_{S_1} & \\
	}
	$$
	with $\Lip(T_1) \le M_{\Lip}^{(p)}(T)$ and $\Lip(S_1) \le M^{\Lip}_{(q)}(S)$.
\end{proposition}

\begin{proposition}
	Let $X$, $Y$ be metric spaces with $Y$ complete and $E$ a Banach lattice, and $1 \le p,q \le \infty$.
	Suppose that $T : X \to E$ is Lipschitz $p$-convex and $S : E \to Y$ is Lipschitz $q$-concave. Then for every $\theta\in(0,1)$,  $ST$ factors through a Banach lattice $U_\theta$ that is $\tfrac{p}{p(1-\theta)+\theta}$-convex and $\frac{q}{1-\theta}$-concave.
\end{proposition}

\begin{corollary}
	Let $E$ be a Banach lattice, $1 \le p,q \le \infty$, and assume that $T : E \to E$ is both Lipschitz $p$-convex and Lipschitz $q$-concave. Then for each $\theta \in (0,1)$, $T^2$ factors through a $\tfrac{p}{p(1-\theta)+\theta}$-convex and $\frac{q}{1-\theta}$-concave Banach lattice. In particular, if $p>1$ and $q< \infty$ then $T$ factors through a super reflexive Banach lattice.
\end{corollary}

A natural question in this context is: if a linear map $T : X \to Y$ between Banach spaces can be factored as a Lipschitz $p$-convex map followed by a Lipschitz $q$-convex one, is there a factorization where the factor maps are in addition linear?
Under certain conditions, the answer is yes.
The following result (and its proof) are in the spirit of
\cite[Thm. 1]{Johnson-Maurey-Schechtman}, \cite[Thm. 2]{Farmer-Johnson-09}, \cite[Thm. 2.1]{Chen-Zheng-12}.

\begin{theorem}\label{thm-equivalence-linear-and-nonlinear}
	Let $T : X \to Y$ be a linear map between a Banach space $X$ and a dual Banach space $Y$, and assume that $T$ admits a factorization $T= T_2 T_1$ where $T_1$ is Lipschitz $p$-convex and $T_2$ is Lipschitz $q$-concave, with $1 \le q < p < \infty$. Then there is also a factorization $T = \tau_2 \tau_1$ where $\tau_1$ is $p$-convex and $\tau_2$ is $q$-concave, and moreover
	$M^{(p)}(\tau_1) \le M_{\Lip}^{(p)}(T_1)$ and $M_{(q)}(\tau_2)\le M^{\Lip}_{(q)}(T_2)$.
\end{theorem}

\begin{proof}
	By Proposition \ref{prop-generalized-krivine}, $T$ can be factored as
	$$
	\xymatrix{
	X \ar[rr]^{T}\ar[d]_{B} &  &Y  \\
	L_p(\mu) \ar[rr]^{i_{p,q}} &   & L_q(\mu) \ar[u]_{A} & \\
	}
	$$
	with $\Lip(B) \le M_{\Lip}^{(p)}(T_1)$ and $\Lip(A) \le M^{\Lip}_{(q)}(T_2)$.
	
	Let's assume first that $X$ is separable. Since $L_p(\mu)$ is reflexive (hence has the Radon-Nikod\'ym property), it follows from \cite[Thm. 6.42]{Benyamini-Lindenstrauss} that $B$ has a point of G\^ateaux differentiability. By translations we can assume that $B$ is G\^ateaux differentiable at 0, and also that $B(0) = 0$. Define $B_n : X \to L_p(\mu)$ and $A_n : L_q(\mu) \to Y$ by
	$B_n(x) = nB(x/n)$, $A_n(z) = nA(z/n)$.
	Since $A \circ i_{p,q} \circ B = T$ is linear, we have that
	$A_n \circ i_{p,q} \circ B_n = T$.
	By the G\^ateaux differentiability condition, there is a linear map $\beta : X \to L_p(\mu)$ such that $\n{ B_n(x) - \beta(x) } \to 0$ for all $x \in X$.
	For each $z \in L_q(\mu)$, let $\tilde{A}(z)$ be the weak$^*$ limit of $(A_n(x))_{n=1}^\infty$ along a fixed free ultrafilter of natural numbers.
	Using the norm convergence of $B_n(x)$ to $\beta(x)$, we deduce that $\tilde{A} \circ i_{p,q} \circ \beta = T$.
	Note that $\n{\beta} \le \Lip(B)$ and $\Lip(\tilde{A}) \le \Lip(A)$.
	Since $\tilde{A}$ is linear on $i_{p,q} \circ \beta(X)$, by \cite[Thm. 6.42]{Benyamini-Lindenstrauss} there is a linear operator $\alpha : L_q(\mu) \to Y$ that coincides with $\tilde{A}$ on $i_{p,q} \circ \beta(X)$ and satisfies $\n{\alpha} \le \Lip(\tilde{A})$.
	The maps $\tau_1 = i_{p,q} \circ \beta$ and $\tau_2 = \alpha$ give the desired factorization.
	
	For a nonseparable Banach space $X$, we can apply the previous argument to the restrictions of $T$ to finite-dimensional subspaces of $X$. Taking an ultraproduct over an ultrafilter on the family of those subspaces, we obtain a factorization
	$$
	\xymatrix{
	X \ar[rr]^{(T_j)_{\mathcal{U}}}\ar[d]_{(\beta_j)_\mathcal{U}} &  &Y_\mathcal{U}  \\
	\big(L_p(\mu_j))_\mathcal{U} \ar[rr]^{(i^j_{p,q})_\mathcal{U}} &   & \big(L_q(\mu_j)\big)_\mathcal{U} \ar[u]_{(\alpha_j)_\mathcal{U}} & \\
	}
	$$
	By taking weak$^*$ limits over the ultrafilter, we can replace $Y_\mathcal{U}$ by $Y$ in the above diagram to get a linear factorization of $T : X \to Y$ through an ultraproduct of the inclusion maps $i^j_{p,q}$.
	By Kakutani's representation theorem, $\big(L_q(\mu_j)\big)_\mathcal{U}$ is lattice isometrically isomorphic to $L_q(\nu)$ for some measure $\nu$. The map $(i^j_{p,q})_\mathcal{U}$ is positive and hence $p$-convex, because $\big(L_p(\mu_j)\big)_\mathcal{U}$is $p$-convex, so by the Maurey/Nikishin factorization theorem $(i^j_{p,q})_\mathcal{U}$ factors through an inclusion $i_{p,q} : L_p(\tilde{\nu}) \to L_q(\tilde{\nu})$, giving the desired result.
\end{proof}

\begin{remark}
When $q=2$ Theorem \ref{thm-equivalence-linear-and-nonlinear} holds for a general Banach space $Y$, because in a Hilbert space all subspaces are 1-complemented.
\end{remark}

\subsection{Factorization for maps that are both Lipschitz $p$-convex and Lipschitz $q$-concave}\label{sec-factorization-theorems-for-both-convex-and-concave}

In  \cite[Sec. 5]{Raynaud-Tradacete} the natural question of factorizations for operators that are both $p$-convex and $q$-concave is studied, and in this section we consider the Lipschitz counterpart.

The motivation is simple: if we have a Lipschitz map $T : E \to F$ between Banach lattices $E$ and $F$ that is both Lipschitz $p$-convex and Lipschitz $q$-concave, we know from Theorems \ref{thm-factorization-lipschitz-p-convex} and \ref{thm-factorization-lipschitz-q-concave} that $T$ can be factored through a $p$-convex Banch lattice and also through a $q$-convex Banach lattice.
Could we have both conditions at the same time?
Note that from the same Theorems, if the map $T$ has a factorization of the form
$$
\xymatrix{
E \ar[r]^{T}\ar[d]_{\phi} &F  \\
E_1 \ar[r]_{R}  &F_1 \ar[u]_{\psi}  
}
$$
where $\phi$ and $\psi$ are positive linear maps, $E_1$ is $q$-concave, $F_1$ is $p$-convex and $R$ is a Lipschitz map, then $T$ is both Lipschitz $p$-convex and Lipschitz $q$-concave.
The question is then: if the map $T : E \to F$ is Lipschitz $p$-convex and Lipschitz $q$-concave, do there exist maps $T_1$ and $T_2$, such that $T = T_1 T_2$, where $T_1$ is Lipschitz $p$-convex and $T_2$ is Lipschitz $q$-concave?

The answer to the corresponding linear question is negative, as shown in \cite[Sec. 5]{Raynaud-Tradacete}, and it turns out that allowing for Lipschitz maps does not change that. Thanks to the following Theorem, the linear examples will also work in the Lipschitz context.

\begin{theorem}\label{thm-nonlinear-implies-linear}
	Let $T : E \to F$ be a linear map between Banach lattices $E$ and $F$, and assume that $T$ admits a factorization $T= T_2 T_1$ where $T_1$ is Lipschitz $q$-concave and $T_2$ is Lipschitz $p$-convex. Then there is also a factorization $T = \tau_2 \tau_1$ where $\tau_1$ is $q$-concave and $\tau_2$ is $p$-convex, and moreover
	$M^{(p)}(\tau_2) \le M_{\Lip}^{(p)}(T_2)$ and $M_{(q)}(\tau_1)\le M^{\Lip}_{(q)}(T_1)$.
\end{theorem}

\begin{proof}
Apply the factorization Theorems \ref{thm-factorization-lipschitz-p-convex} and \ref{thm-factorization-lipschitz-q-concave} to obtain
$$
\xymatrix{
E \ar[rr]^{T_1}\ar[dr]_{\phi} & &G \ar[rr]^{T_2}\ar[dr]_{R} & &F  \\
 &V \ar[ru]_{S}  &  &W \ar[ru]_{\psi} & \\
}
$$
where $G$, $W$ and $V$ are Banach lattices with $W$ (resp. $V$) being $p$-convex (resp. $p$-concave) with constant 1; 
$R$ and $S$ are Lipschitz maps with constant at most $1$, and
$\psi$ (resp. $\phi$) is a lattice
homomorphism of norm at most $M_{\Lip}^{(p)}(T_2)$ (resp. $M^{\Lip}_{(q)}(T_1)$ ) that is injective (resp. has dense range).
Let $J = RS$. We claim that $J$ is a linear map.
Given $u,v \in V$, choose sequences $(u_n)_{n=1}^\infty, (v_n)_{n=1}^\infty$ in $V$ such that $\phi(u_n) \to u$ and $\phi(v_n)\to v$.
Since $\phi$ is linear and continuous, $\phi(u_n+v_n) \to u+v$.
$T$ is also linear, so $\psi J \phi(u_n+v_n) = \psi J \phi(u_n) + \psi J \phi(v_n) = \psi(J\phi(u_n) + J\phi(v_n))$.
We conclude that $\psi J(u+v) = \psi(J(u) + J(v))$, and by the injectivity of $\psi$, $J$ is linear.
Choosing $\tau_1 = J \phi$ and $\tau_2 = \psi$, we get that 
$\tau_1$ is $q$-concave and $\tau_2$ is $p$-convex, and in addition
	$M^{(p)}(\tau_2) \le M_{\Lip}^{(p)}(T_2)$ and $M_{(q)}(\tau_1)\le M^{\Lip}_{(q)}(T_1)$.
\end{proof}

Theorem \ref{thm-nonlinear-implies-linear} and the results from \cite[Sec. 5]{Raynaud-Tradacete} immediately imply the following Corollary.

\begin{corollary}
	\begin{enumerate}[(a)]
		\item Let $T : E \to F$ be a linear map from an $\infty$-convex Banach lattice (an AM-space) $E$ to a $q$-concave Banach lattice $F$ ($q < \infty$). If $T$ can be factored as $T = SR$, with $R$ Lipschitz $q$-concave and $S$ Lipschitz $\infty$-convex, then $T$ is compact.
		\item The formal inclusion $T : C(0,1) \to L_q(0,1$) is Lipschitz $q$-concave and Lipschitz $\infty$-convex, but it does not factor as $T = T1 \circ T2$, with $T_1$ Lipschitz $\infty$-convex, and $T_2$ Lipschitz $q$-concave.
		\item Let $T : E \to F$ be a linear map from a $p$-convex Banach lattice $E$ to a $1$-concave Banach lattice (an AL-space) $F$. If $T$ can be factored as $T = SR$, with $R$ Lipschitz 1-concave and $S$ Lipschitz $p$-convex, then $T$ is compact.
		\item If a lattice homomorphism $T : L_p(0,1) \to Lq(0,1)$ ($1 < q < p < \infty$) can be factored as $T = T_2 \circ T_1$ with $T_1$ Lipschitz $q$-concave and $T_2$ Lipschitz $p$-convex, then $T$ is AM-compact (recall that a linear map $T : E \to X$ between a Banach lattice $E$ and a Banach space $X$ is called \emph{AM-compact} if $T[-x,x]$ is relatively compact for every positive $x \in E$.).
		\item 	For $1 < q < p < \infty$, the formal inclusion $i_{p,q} : L_p(0,1)\to L_q(0,1)$ cannot be factored as $i_{p,q} = T_2 \circ T_1$ with $T_1$ Lipschitz $q$-concave and $T_2$ Lipschitz $p$-convex.
	\end{enumerate}
\end{corollary}

Even though the na\"ive factorization scheme for linear maps that are both $p$-convex and $q$-concave did not work out, Raynaud and Tradacete were able to prove that if one ``gives up'' a little on the exponents of convexity and concavity involved, one still gets such a factorization \cite[Thm. 15]{Raynaud-Tradacete}.
The following would be a Lipschitz version of that result. 

\begin{question}
	Suppose that a Lipschitz map $T: E \to F$ between Banach lattices
	is Lipschitz $p$-convex and Lipschitz $q$-concave, with $1 < p \le \infty$ and $1 \le q < \infty$.
	Can we find $1 < p_0 < p$ and $q < q_0 < \infty$ so that there is a factorization of $T$ as
	$$
	\xymatrix{
	E \ar[r]^{T}\ar[d]_{\phi} &F  \\
	E_0 \ar[r]_{R}  &F_0 \ar[u]_{\varphi}  
	}
	$$
	where $\phi$ and $\varphi$ are positive linear maps, $E_0$ is $p_0$-convex, $F_0$ is $q_0$-concave and $R$ is a Lipschitz map?
	Moreover: given $\theta\in(0,1)$, can we have $p_0 = \frac{p}{\theta + (1-\theta)p}$ and $q_0 = \frac{q}{1-\theta}$?
\end{question}

\begin{remark}
	The proof from \cite[Thm. 15]{Raynaud-Tradacete} cannot be easily adapted to the Lipschitz context.
	The arguments in that proof are heavily based on complex interpolation because that method works very well for lattices, as pointed out by Calder\'on \cite{Calderon}.
	Complex interpolation, however, is not well suited to work with Lipschitz maps.
	The results available, e.g. the main one from \cite{Bergh}, require strong extra assumptions due to the fact that a Lipschitz function generally does not preserve analyticity. 
\end{remark}

\section*{Acknowledgments}
The author thanks Professor W. B. Johnson for many helpful discussions and suggestions.

\bibliographystyle{amsalpha}
\def\cprime{$'$}
\providecommand{\bysame}{\leavevmode\hbox to3em{\hrulefill}\thinspace}
\providecommand{\MR}{\relax\ifhmode\unskip\space\fi MR }
\providecommand{\MRhref}[2]{%
  \href{http://www.ams.org/mathscinet-getitem?mr=#1}{#2}
}
\providecommand{\href}[2]{#2}

\end{document}